\UseRawInputEncoding
\documentclass{amsart}
\usepackage{amsmath,amssymb}
\usepackage{color}
\usepackage{hyperref}
\numberwithin{equation}{section}
\newtheorem{theorem}{Theorem}[section]
\newtheorem{lemma}{Lemma}
\newtheorem{corollary}[lemma]{Corollary}
\numberwithin{lemma}{section}
\theoremstyle{definition}
\newtheorem{definition}{Definition}
\theoremstyle{remark}
\newtheorem{remark}{Remark}
\renewcommand{\vec}[1]{\boldsymbol{#1}} 

\begin{document}
\date{}
\title[Co-circular C.C.]{On  the centered co-circular central configurations for the n-body problem}

 \author{Zhiqiang Wang}
\address{College of Mathematics and Statistics,  Chongqing University, Chongqing 401331, P.R.China}
\curraddr{}
\email{wzq2015@cqu.edu.cn}

\begin{abstract}
For the power-law potential  $n$-body problem, we study a special kind of central configurations where all the masses lie on a circle and the center of mass coincides with the center of the circle.
It is also called the centered co-circular central configuration. We get some symmetry results for  such central configurations.
We show that for positive numbers $\alpha>0$ and  integers $n\geq3$ satisfying $\frac{1}{n}\sum_{j=1}^{n-1}\csc^{\alpha}\frac{j\pi}{n}\leq1+\frac{\alpha}{4}$, the regular $n$-gon  with equal masses is the unique centered co-circular central configuration for the $n$-body problem with  power-law potential $U_{\alpha}$.  It quickly follows that for the Newtonian $n$-body problem (in the case $\alpha=1$) and $n\leq6$, the regular $n$-gon   is the unique centered co-circular central configuration.

\end{abstract}
\maketitle

\section{Introduction and the Main Result}
\label{intro}
The  $n$-body problem concerns the  motion of $n$ point masses $m_k>0$ and their positions $x_k\in \mathbb{R}^{3}$, $k=1,...,n$. 
For $n\geq2$ and $\alpha\geq0$, this motion can be  governed by a power-law potential:
\begin{equation*}
m_k\ddot{x}_k=\frac{\partial U_{\alpha}(\vec{x})}{\partial x_k},\quad k=1,2,\ldots,n ,
\end{equation*}
where $\vec{x}=(x_1,x_2,\ldots,x_n)^T$, and in a slight abuse of notation, 
$$U_{\alpha}(\vec{x})=
\left\{ \begin{aligned}
&\sum\limits_{1\leq j<k\leq n}\frac{m_jm_k}{|x_j-x_k|^{\alpha}},&\alpha\neq0,\\
&\sum\limits_{1\leq j<k \leq N}m_jm_k\log |x_{j}-x_{k}|, & \alpha=0.
\end{aligned}
\right.
$$
When $\alpha=1$, it is the classical Newtonian $n$-body problem,
and for the limiting case $\alpha=0$, it corresponds to an N-Vortex Problem.
In order to avoid  collisions, we consider the configuration space
$${\bold{X}}\backslash\triangle=\{(x_1,x_2,\ldots,x_n)^T\in\mathbb{R}^{3n}:x_j\neq x_k \quad \text{when} \quad j\neq k\},$$
and give the following definition of  central configurations. 
\begin{definition}
A configuration $\vec{x}=(x_1,x_2\ldots,x_n)^{T}\in  {\bold{X}}\backslash\triangle$ is  called a central configuration (CC for short) 
 for masses $\vec{m}=(m_1,m_2,\dots,m_{n})^{T}$  if there exists some constant $\lambda \in \mathbb{R}$ such that
\begin{equation}\label{eqcc0}
\sum_{\substack{j\neq k \\ j=1}}^n\frac{x_j-x_k}{|x_j-x_k|^{\alpha+2}}m_jm_k=-\frac{\lambda}{\alpha} m_k(x_k-c_0),\quad k=1,2,...,n ,
\end{equation}
or simply
$$\frac{\partial U_{\alpha}(\vec{x})}{\partial x_k}=-\lambda m_k(x_k-c_0),\quad k=1,2,...,n ,$$
where the center of mass is defined as
$c_0=\frac{\scriptstyle
\sum_{k=1}^nm_kx_k}{
\scriptstyle
\sum_{k=1}^n m_k}.$
\end{definition}
Central configurations play ``central'' role in the  $n$-body problem \cite{saari1980} such as in the analysis of collision orbits or expanding gravitational systems. It also gives rise to a family of simple   and explicit  periodic solutions \cite{moeckel1990}.
For instance, a planar central configuration can naturally give rise to a relative equilibrium motion where every mutual distance $r_{jk}=|x_{j}-x_{k}|$ stays constant.

The equations for central configurations \eqref{eqcc0} are invariant under rotations, translations and dilations. Thus we always consider equivalent classes of central configurations modulo these symmetries.
The study of Newtonian central configurations dates back to Euler \cite{euler} and Lagrange \cite{lagrange} who studied the  3-body case.
There are only five classes of central configurations for the Newtonian 3-body problem, three Euler collinear CCs and two Lagrange equilateral triangle CCs.
Although centuries have passed since then, a complete understanding of CCs remains elusive. 
When $n>3$, it is almost impossible to solve the equations of central configurations \eqref{eqcc0} for any given masses. 
A longstanding problem, known as the finiteness problem,  is for any given positive masses whether there are only finitely many equivalent classes of CCs.
It is proposed by Wintner  \cite{Wintner} and  listed by Smale \cite{Smale}  as the sixth problem on his list of problems for this century. 
Albouy et al. \cite{albouy2012some} also list it as the 9th open problem in celestial mechanics.
For the collinear case, Moulton \cite{moulton} showed that there are exactly $n!/2$ equivalent classes of CCs for the $n$-body problem.
Most recently, using BKK theory in algebraic geometry, Hampton and Moeckel \cite{HM2006} proved the finiteness of the number of  equivalent classes of CCs for n = 4.
For 5 bodies in the plane, Albouy and Kaloshin \cite{AK2012} showed the 
 the number is finite, except perhaps if the masses belong to an explicit codimension 2 subvariety of the space of positive masses.
For the spatial 5-body central configuration, Hampton and Jenson \cite{HJ2011} showed the finiteness,  apart from some explicitly given special cases of mass values.

In this paper, we study a special kind of central configuration where all the masses lie on a circle and the center of mass coincides with the center of the circle.
In what follows we call it \emph{centered co-circular central configurations}  as in \cite{Hampton16}.
There are also many work on the co-circular central configurations which is not centered (see  \cite{cor-roberts} \cite{santoprete2021}and  references therein).
There is an open question asked by Chenciner \cite{chenciner04} when he studied the choreography solutions in the $n$-body problem.

\begin{center}
\emph{Is the regular $n$-gon with equal masses the sole central configuration such that
all the bodies lie on a circle and
the center of mass coincides with the center of the circle?}
\end{center}

The regular $n$-gon central configuration with equal masses is the trivial  one. It gives rise to the trivial choreography solution where each body moves  (with equal time spacings) on the same circle.
If there were some other  centered co-circular central configurations we can immediately have a choreography solution with unequal time spacings between the bodies and maybe unequal masses.
Chenciner's question is also listed as Problem 12 in a collection of important open problems in celestial mechanics\cite{albouy2012some}.
The first one who answered  the question was Hampton. He proved that the only four-body centered co-circular CC  is the square with equal masses \cite{HAMPTON2005}. 
Llibre and Valls \cite{Llibre2015} studied the  case $n = 5$. 
All the above results are just in the Newtonian case $\alpha=1$.
A big progress is the work of  Cors et al.\cite{CORS201444} who study the general co-circular configurations with power-law potential $U_{\alpha} (\alpha\geq0)$. They proved that, for any choice of positive masses, if such a central configuration exists, it is unique. It  quickly follows that when the masses are all equal,  the only solution is the regular $n$-gon. 
In that paper, they also give a positive answer to Chenciner's question for the power law potential $U_{0}$ (in the case $\alpha=0$).
They proved that, in the n-vortex problem  the only possible centered co-circular central configuration is the regular $n$-gon with equal vorticities.
Our  main result is  
\begin{theorem}\label{main1}
For all the integers $n\geq3$ and real numbers $\alpha>0$ satisfying 
$$\frac{1}{n}\sum_{j=1}^{n-1}\csc^{\alpha}\frac{j\pi}{n}\leq1+\frac{\alpha}{4},$$
 the regular $n$-gon CC with equal masses is the unique central configuration for power-law potential $U_{\alpha}$, such that all the bodies lie on a circle, and the center of mass coincides with the center of the circle.
\end{theorem}
We notice that the function $\frac{1}{n}\sum_{j=1}^{n-1}\csc^{\alpha}\frac{j\pi}{n}$ is increasing  with respect to $n$ and $\alpha$, separately.
For any given integer $n>0$, we can get a positive answer to Chenciner's question for the $n$-body problem with power-law potential $U_{\alpha}$ in the case $\alpha=\alpha(n)$ is  sufficiently small.
For the classical  Newtonian case $\alpha=1$, a direct computation shows that 
$$\frac{1}{6}\sum_{j=1}^{6-1}\csc\frac{j\pi}{6}=\frac{5}{6}+\frac{2\sqrt{3}}{9}<1+\frac{1}{4}.$$
\begin{corollary}
In the Newtonian $n(\leq6)$ body problem, the regular $n$-gon CC with equal masses is the unique central configuration, such that all the bodies lie on a circle, and the center of mass coincides with the center of the circle.
\end{corollary}

This paper is organized as follows: In Section \ref{sec2} we derive the equations for the centered co-circular CC. Following the approach of  Cors et al.\cite{CORS201444}, we can see  the centered co-circular CC as the unique minimum of a new function for any given positive masses (see Lemma \ref{unique1}).
In Section \ref{sec3}, we provide some symmetry results and lemmas of  centered co-circular CCs by  the  uniqueness of minimum.
We also give a positive answer to Chenciner's question for any power-law potential in the 3 and 4-body problem. 
In Section \ref{sec4} we prove Theorem \ref{main1} based on the lemmas in Section \ref{sec3}.
\section{Equations for centered co-circular central configurations}\label{sec2}
We suppose that  $n$ positive masses $\vec{m}=(m_{1},m_{2},\dots,m_{n})^{T}\in\mathbb{R}_{+}^{n}$ lie on a unit circle centered at origin in the complex plane.
Their positions are  $\vec{q}=(q_{1},q_{2},\dots,q_{n})^{T}\in  \mathbb{C}^{n}$ where $q_{j}=e^{\sqrt{-1}\theta_{j}}=\cos \theta_{j}+\sqrt{-1}\sin \theta_{j}$.
Without loss of generality, let $\theta_{j}\in (0,2\pi]$ and in order to avoid collisions,  we suppose
$$\vec{\theta}=(\theta_{1},\theta_{2},\dots,\theta_{n})^{T}\in \mathcal{K}=\{\vec{\theta}\in\mathbb{R}^{n}: 0< \theta_{1}<\theta_{2}<\dots<\theta_{n}\leq 2\pi\}.$$
Thus when the mass vector $\vec{m}$ is given, their order is also determined. 
Under these notations, 
the potential $U_{\alpha}$ can be written as
 $$U_{\alpha}(\vec{m},\vec{\theta})=\sum_{j<k}\frac{m_{j}m_{k}}{r_{jk}^{\alpha}(\vec{\theta})}=2^{-\alpha}\sum_{j<k}\frac{m_{j}m_{k}}{|\sin^{\alpha}\frac{\theta_{k}-\theta_{j}}{2}|},$$
where 
$$r_{jk}=|q_{j}-q_{k}|=|2\sin \frac{\theta_{j}-\theta_{k}}{2}|=\sqrt{2-2\cos(\theta_{j}-\theta_{k})}$$
is the mutual distance  between $q_{j}$ and $q_{k}$.
Then by \eqref{eqcc0}, they form a centered co-circular CC
if and only if 
\begin{equation*}\label{eqcc1}
\left\{\begin{aligned}
&\sum_{j\neq k}^{n}\frac{m_{j}(1-q_{j}/q_{k})}{r_{jk}^{\alpha+2}}=\frac{\lambda}{\alpha},\ \ k=1,2,\dots,n,\\
&\sum_{j=1}^{n}m_{j}q_{j}=0.
\end{aligned}\right.
\end{equation*}
Actually the last equations can also be deduced from the first $n$ equations by multiplying each of them by $m_{k}q_{k}$ and taking summation. 
Writing their real and imaginary part separately, we have
$$\left\{
\begin{aligned}
&\alpha m_{k}\sum_{j\neq k}^{n}\frac{m_{j}}{r_{jk}^{\alpha+2}}\sin(\theta_{j}-\theta_{k})=0,\\
&\sum_{j\neq k}^{n}\frac{m_{j}}{r_{jk}^{\alpha}}=\tilde{\lambda},\\
&\sum_{j=1}^{n}m_{j}\cos\theta_{j}=\sum_{j=1}^{n}m_{j}\sin\theta_{j}=0,
\end{aligned}
\right.
$$
where $\tilde{\lambda}=\frac{2\lambda}{\alpha} {=\frac{2U_{\alpha}}{M}}$.
Seeing $U_{\alpha}(\vec{m},\vec{\theta})=\sum_{j<k}\frac{m_{j}m_{k}}{r_{jk}^{\alpha}(\vec{\theta})}$ as a function of varibles $\vec{m},\vec{\theta}$, 
and noticing that
$$U_{-2}(\vec{m},\vec{\theta})=\sum_{j<k}m_{j}m_{k}r_{ij}^{2}=\sum_{j<k}m_{j}m_{k}(2-2\cos(\theta_{j}-\theta_{k})),$$
these equations are equivalent to
\begin{equation*}
\frac{\partial}{\partial \theta_{k}}U_{\alpha}=0,
\ \ \ 
\frac{\partial}{\partial m_{k}}U_{\alpha}=\tilde{\lambda}, 
\ \ \ 
k=1,2,\dots,n,
\end{equation*}
\begin{equation*}
\frac{\partial}{\partial \theta_{k}}U_{-2}=0,
\ \ \ 
\frac{\partial}{\partial m_{k}}U_{-2}=2M, 
\ \ \ 
k=1,2,\dots,n,
\end{equation*}
or more compactly 
\begin{equation}\label{ccc}
\nabla_{\vec{\theta}} U_{\alpha}|_{(\vec{m},\vec{\theta})}=\vec{0},
\ \ \ 
\nabla_{\vec{m}} U_{\alpha}|_{(\vec{m},\vec{\theta})}=\tilde{\lambda}\vec{1},
\end{equation}
\begin{equation}\label{ccc1}
\nabla_{\vec{\theta}} U_{-2}|_{(\vec{m},\vec{\theta})}=\vec{0},
\ \ \ 
\nabla_{\vec{m}} U_{-2}|_{(\vec{m},\vec{\theta})}=2M\vec{1},
\end{equation}
where the vectors $\vec{0}=(0,0,\dots,0)^{T}$ and $\vec{1}=(1,1,\dots,1)^{T}$.
Thus $\vec{m},\vec{\theta}$ form a centered co-circular central configuration if and only if 
the equations \eqref{ccc} and \eqref{ccc1} hold
 (if and only if 
the equations \eqref{ccc} hold). 
So if $\vec{m},\vec{\theta}$ satisfy \eqref{ccc} and \eqref{ccc1}, we simply denote 
$$(\vec{m},\vec{\theta})\in \mathcal{CC}.$$
Cors et al.\cite{CORS201444}  notice that $\nabla_{\vec{\theta}} U_{\alpha}=\vec{0}$ means  $\vec{\theta}$ is a critical points of $U_{\alpha}$ when the positive mass vector $\vec{m}$ is given. 
 By computing the corresponding Hessian matrix of $U_{\alpha}$, they showed that this critical point is a minimum and is unique up to translation which can be removed by specifying $\theta_{n}=2\pi$.
Let
$$\mathcal{K}_{0}=\{\vec{\theta}\in\mathcal{K}|\theta_{n}=2\pi\},\ \ \ \ 
\mathcal{CC}_{0}=\{(\vec{m},\vec{\theta})\in \mathcal{CC}|\vec{\theta}\in\mathcal{K}_{0}\},$$
their result can be expressed as follows:
\begin{lemma}\label{unique}
For any given positive mass vector $\vec{m}\in\mathbb{R}_{+}^{n}$, there exists a unique $\vec{\varphi}_{\vec{m}}\in\mathcal{K}_{0}$ such that
 $\nabla_{\vec{\theta}} U_{\alpha}{(\vec{m},\vec{\varphi}_{\vec{m}})}=\vec{0}$. Moreover, this critical point is a minimum.
\end{lemma}

They also figure out that for any given positive mass vector $\vec{m}$,  a centered co-circular central configuration  (if exists), namely  $\vec{\theta}$, should be the unique minimum of the function $U_{\alpha}$. That is to say 
$$(\vec{m},\vec{\theta})\in \mathcal{CC}_{0}\Rightarrow\vec{\theta}=\vec{\varphi_{m}}.$$
 Let $$f_{K}(\vec{m},\vec{\theta})=U_{\alpha}+\frac{U_{-2}}{K}, $$ 
 where $K\geq \frac{2^{3+\alpha}}{\alpha}$ is some constant, we can extend the above lemma a little bit: 
\begin{lemma}\label{unique1}
For any given positive mass vector $\vec{m}\in\mathbb{R}_{+}^{n}$ and  real number $K\geq \frac{2^{3+\alpha}}{\alpha}$, there exists a unique $\vec{\theta}_{\vec{m}}\in\mathcal{K}_{0}$ such that
 $\nabla_{\vec{\theta}} f_{K}{(\vec{m},\vec{\theta}_{\vec{m}})}=\vec{0}$. Moreover, this critical point is a minimum and $(\vec{m},\vec{\theta})\in \mathcal{CC}_{0}$ implies $\vec{\theta}=\vec{\theta_{m}}$.
\end{lemma}
\begin{proof}
For  any given number $\beta\neq0$, we compute that
$$\left\{\begin{aligned}
&\frac{\partial^{2}}{\partial \theta_{i}\partial \theta_{j}}U_{\beta}=\frac{-\beta m_{i}m_{j}(1+\beta \cos^{2}\frac{\theta_{j}-\theta_{i}}{2})}{|2\sin\frac{\theta_{j}-\theta_{i}}{2}|^{\beta+2}},&if\ i\neq j\\
&\frac{\partial^{2}}{\partial \theta_{i}^{2}}U_{\beta}=-\sum_{j\neq i}^{n}\frac{\partial^{2}}{\partial \theta_{i}\partial \theta_{j}}U_{\beta}.
\end{aligned}\right.$$
So we have 
$$\frac{\partial^{2}}{\partial \theta_{i}^{2}}f_{K}=-\sum_{j\neq i}^{n}\frac{\partial^{2}}{\partial \theta_{i}\partial \theta_{j}}f_{K},$$
and for $i\neq j$
\begin{align*}
\frac{\partial^{2}}{\partial \theta_{i}\partial \theta_{j}}f_{K}
&=m_{i}m_{j}[\frac{-\alpha (1+\alpha \cos^{2}\frac{\theta_{j}-\theta_{i}}{2})}{|2\sin\frac{\theta_{j}-\theta_{i}}{2}|^{\alpha+2}}+\frac{2-4\cos^{2}\frac{\theta_{j}-\theta_{i}}{2}}{K}]\\
&\leq m_{i}m_{j}[\frac{-\alpha (1+\alpha \cos^{2}\frac{\theta_{j}-\theta_{i}}{2})}{2^{\alpha+2}}+\frac{2-4\cos^{2}\frac{\theta_{j}-\theta_{i}}{2}}{K}].
\end{align*}
When $K\geq \frac{2^{3+\alpha}}{\alpha}$, we have 
$$\frac{\partial^{2}}{\partial \theta_{i}\partial \theta_{j}}f_{K}\leq m_{i}m_{j}[\frac{-\alpha (1+\alpha \cos^{2}\frac{\theta_{j}-\theta_{i}}{2})}{2^{\alpha+2}}+\frac{\alpha}{2^{2+\alpha}}-\frac{4\cos^{2}\frac{\theta_{j}-\theta_{i}}{2}}{K}]\leq0.$$
Thus the Hessian matrix $D^{2}f_{K}=(\frac{\partial^{2}f_{K}}{\partial \theta_{i}\partial \theta_{j}})_{n\times n}$  is  {diagonally dominant}. So the quadratic form $u^{T}(D^{2}f_{K})u\geq0$, where the equality holds if and only if $u$ is a scalar multiple of $\vec{1}$ (see Remark \ref{remark1} below).
This vector corresponds to the translational invariance of $f_{K}$, which can be removed by specifying $\theta_{n}=2\pi$.
Since $\lim_{\vec{\theta}\rightarrow \partial\mathcal{K}_{0}}f_{K}(\vec{m},\vec{\theta})=+\infty$, there exists a unique critical point (which is a minimum) $\vec{\theta}_{\vec{m}}\in\mathcal{K}_{0}$ of 
 $ f_{K}$.

So if  $(\vec{m},\vec{\theta})\in \mathcal{CC}_{0}$ is a centered co-circular CC, \eqref{ccc} and \eqref{ccc1} are satisfied and thus $\vec{\theta}=\vec{\theta_{m}}$ is the unique minimum  of the  function  $f_{K}, K\geq \frac{2^{3+\alpha}}{\alpha}$. 
We point out that because \eqref{ccc1} can be deduced from \eqref{ccc},  $\vec{\varphi_{m}}$ in Lemma \ref{unique} is equal to $\vec{\theta_{m}}$ in this lemma.
%

\end{proof}
\begin{remark}\label{remark1}
We say $A=(a_{ij})_{n\times n}$ is a   diagonally dominant matrix, if $|a_{ii}| \geq \sum_{j\neq i} |a_{ij}|$  for $i=1,2,\dots,n$. A symmetric diagonally dominant real matrix  with nonnegative diagonal entries is positive semidefinite.
Moreover, if the off-diagonal entries of the symmetric real matrix $A$ are non-positive  and $a_{ii} =- \sum_{j\neq i} a_{ij}>0$  for $i=1,2,\dots,n$ (i.e., $A\vec{1}=0$), 
then $A$ is positive semidefinite and  the vector  $\vec{1}$ is its unique zero eigenvector.
To see the uniqueness, we suppose there is another vector $\vec{v}=(v_{1},\dots,v_{n})^{T}$ not parallels to $\vec{1}$ such that $A\vec{v}=\vec{0}$.
Then we can choose $j_{0}$ such that $v_{j_{0}}\geq v_{j},\forall j=1,2,\dots,n$ and $v_{j_{0}}> v_{j_{1}}$ for some $j_{1}$. 
It is not difficult to check that the $j_{0}$-th element of $A\vec{v}$ is strictly greater than zero, contradicts with $A\vec{v}=\vec{0}$.

\end{remark}

\section{Symmetry of centered co-circular central configurations}\label{sec3}

In this section, we study the symmetry of centered co-circular CCs.
It is  intuitional that a reflection or a cyclic permutation of the co-circular CC is also a co-circular CC. More precisely, if we denote
$$P=
         \begin{pmatrix}
         0 & 1 & 0 & \ldots & 0& 0 \\
         0 & 0 & 1 & \ldots & 0& 0 \\
         . & . & . & \ldots & .& . \\
          0 & 0 & 0 & \ldots & 0& 1 \\
         1 & 0 & 0 & \ldots & 0& 0
                 \end{pmatrix}     \ 
S=
         \begin{pmatrix}
         0 & 0  & \ldots & 0 & 1& 0 \\
         0 & 0  & \ldots & 1& 0& 0 \\
         . & . &  \ldots & . &.& . \\
          1 & 0  & \ldots& 0 & 0& 0 \\
         0 & 0  & \ldots & 0& 0& 1
                 \end{pmatrix},$$
and let $G=<P,S>$   be the dihedral group generated by the matrix $P$ and $S$.             
Then for all $ g \in G $, we have
$$(\vec{m},\vec{\theta})\in \mathcal{CC}\Leftrightarrow(g\vec{m},g\vec{\theta})\in \mathcal{CC}.$$
This is because  \eqref{ccc} and \eqref{ccc1} are invariant if we take the same permutation on indices of $\vec{m}$ and $\vec{\theta}$ simultaneously.
Moreover, to restrict the angles $\vec{\theta}\in \mathcal{K}_{0}$, we can let           
$$\mathcal{P}=
         \begin{pmatrix}
         -1 & 1 & 0 & \ldots & 0& 0 \\
         -1 & 0 & 1 & \ldots & 0& 0 \\
         . & . & . & \ldots & .& . \\
          -1 & 0 & 0 & \ldots & 0& 1 \\
         0 & 0 & 0 & \ldots & 0& 1
                 \end{pmatrix}\       
\mathcal{S}=
         \begin{pmatrix}
         0 & 0  & \ldots & 0 & -1& 1 \\
         0 & 0  & \ldots & -1& 0& 1 \\
         . & . &  \ldots & . &.& . \\
          -1 & 0  & \ldots& 0 & 0& 1 \\
         0 & 0  & \ldots & 0& 0& 1
                 \end{pmatrix},$$       
then $\vec{\theta}\in \mathcal{K}_{0}$   if and only if $\mathcal{P}^{h}\mathcal{S}^{l}\vec{\theta}\in \mathcal{K}_{0}, \forall h,l\in \mathbb{Z}.$   
For any $g=P^{h}S^{l}\in G$, letting $\hat{ g}=\mathcal{P}^{h}\mathcal{S}^{l}$, we can define the group representation         
$$g\cdot (\vec{m},\vec{\theta})=(g\vec{m},\hat{ g}\vec{\theta}).$$ 
 We also recall that for given positive mass vector $\vec{m}$,  the unique minimum of $f_{K}|_{\mathcal{K}_{0}}$
   always exists (see Lemma \ref{unique1}), and we denote it as  $\vec{\theta}_{\vec{m}}\in\mathcal{K}_{0}$. 
It will  be a centered co-circular CC if it also satisfies the second part of equations \eqref{ccc} and \eqref{ccc1}.

\begin{lemma}\label{SPCC}
Suppose $(\vec{m},\vec{\theta_m})\in \mathcal{CC}_{0}$ form a centered co-circular central configuration, then
\begin{enumerate}
\item
 $g\cdot(\vec{m},\vec{\theta_m})\in \mathcal{CC}_{0}, \forall g\in G$;
 \item  $f_{K}(\vec{m},\vec{\theta_m})=f_{K}(g\vec{m},\hat{g}\vec{\theta_m})\leq f_{K}(g\vec{m},\vec{\theta_m})$ and $\hat{g}\vec{\theta_m}=\vec{\theta_{gm}}$;
 \item 
  $\vec{m}=g\vec{m}$ implies  $\vec{\theta_m}=\hat{g}\vec{\theta_m}$.
  \end{enumerate}
 \end{lemma}                 
  \begin{proof}
For $\vec{\theta}\in\mathcal{K}_{0}$ which means $0< \theta_{1}<\theta_{2}<\dots<\theta_{n}= 2\pi$, we have
  $$\mathcal{P}\vec{\theta}=(\theta_{2}-\theta_{1},\theta_{3}-\theta_{1},\dots,\theta_{n}-\theta_{1},\theta_{1}-\theta_{1}+2\pi)\in\mathcal{K}_{0},$$
  $$\mathcal{S}\vec{\theta}=(2\pi-\theta_{n-1},2\pi-\theta_{n-2},\dots,2\pi-\theta_{1},2\pi)\in\mathcal{K}_{0}.$$
They are the rotation or reflection on indices of the configuration $\vec{\theta}$ plus a translation.  
Since \eqref{ccc} and \eqref{ccc1} are invariant by this group action, we have
$$(\vec{m},\vec{\theta})\in \mathcal{CC}_{0}\Leftrightarrow(P\vec{m},\mathcal{P}\vec{\theta})\in \mathcal{CC}_{0}\Leftrightarrow (S\vec{m},\mathcal{S}\vec{\theta})\in \mathcal{CC}_{0} \Leftrightarrow g\cdot(\vec{m},\vec{\theta})\in \mathcal{CC}_{0}.$$
Then $$f_{K}(\vec{m},\vec{\theta_m})=f_{K}(g\vec{m},\hat{g}\vec{\theta_m})$$ follows and the uniqueness of the minimum implies that $\hat{g}\vec{\theta_m}=\vec{\theta_{gm}}$.
Applying Lemma \ref{unique1}, we have the following inequality 
  $$f_{K}(\vec{m},\vec{\theta_m})=f_{K}(g\vec{m},\hat{g}\vec{\theta_m})\leq f_{K}(g\vec{m},\vec{\theta_m}),$$
  where the equality holds if and only if $\vec{\theta_m}=\hat{g}\vec{\theta_m}$. 
  So if $\vec{m}=g\vec{m}$, we will always have the equality, then $\vec{\theta_m}=\hat{g}\vec{\theta_m}$. This finishes the proof.
\end{proof}    
This lemma tells us that,  for a centered co-circular CC, the symmetry of the distribution of the masses can imply the the symmetry of the configuration. For example, if $(\vec{m},\vec{\theta_m})\in \mathcal{CC}_{0}$ and  $\vec{m}=P\vec{m}$ (i.e., all the masses are equal),  we will have  $\vec{\theta_m}=\mathcal{P}\vec{\theta_m}$, which is just the equal-mass regular $n$-gon CC.

\begin{corollary}\label{em}
For all  integers $n\geq3$ and real numbers $\alpha>0$, if all the masses are equal, then
the unique centered co-circular central configuration for power-law potential $U_{\alpha}$  is the equal-mass regular $n$-gon central configuration.
\end{corollary}

%

Let the matrix
$$H_{\alpha,\vec{{m}}}=(r_{jk}^{-\alpha}(\vec{\theta_{m}})+\frac{1}{K}r_{jk}^{2}(\vec{\theta_{m}}))_{n\times n}$$
which is determined by the mass vector $\vec{m}$ (seeing $r_{ij}(\vec{\theta})$ as indirect variables)
and consider the function
$$h_{\vec{m}}(\vec{y})=f_{K}(\vec{{y},{\theta}_{{m}}})
=\sum_{j<k}(\frac{y_{j}y_{k}}{r_{jk}^{\alpha}(\vec{\theta})}+\frac{y_{j}y_{k}r_{jk}^{2}(\vec{\theta})}{K})
=\frac{1}{2}\vec{y}^{T}H_{\alpha,\vec{{m}}}\vec{y}, \  \ \  \vec{y}\in \mathbb{R}^{n}.$$
Then $h_{\vec{m}}(\vec{y})$ is a homogeneous polynomial of degree 2.

 If $\vec{m}$ and $\vec{\theta_{m}}$ also satisfy the second part of \eqref{ccc} and  \eqref{ccc1}, i.e., $(\vec{{m},{\theta}_{{\vec{m}}}})\in \mathcal{CC}_{0}$ is a centered co-circular CC,  we have
$$Dh_{\vec{m}}(\vec{m})=\nabla_{\vec{m}}f_{K}(\vec{{m},{\theta}_{{m}}})=(\tilde{\lambda}+\frac{2M}{K})\vec{1}.$$
Writing $h_{\vec{m}}(\vec{y})$ in its Taylor series at the point ${\vec{m}}$, we get
\begin{equation*}\label{taylor}
h_{\vec{m}}(\vec{y})=f_{K}(\vec{{m},{\theta}_{{m}}})+\sum_{j=1}^{n}(\tilde{\lambda}+\frac{2M}{K})(y_{j}-{m}_{j})+
\frac{1}{2}(\vec{y-m})^{T} H_{\alpha,\vec{{m}}} (\vec{y-m}),
\end{equation*}
where $H_{\alpha,\vec{{m}}}$ happens to be  the Hessian of $h_{\vec{m}}(\vec{y})$ at the point $\vec{{m}}$.

So $(\vec{{m},{\theta}_{{\vec{m}}}})\in \mathcal{CC}_{0}$ and $\sum_{j=1}^{n}y_{i}=\sum_{j=1}^{n}m_{i}$ would imply
\begin{equation}\label{eq4}
f_{K}(\vec{{y},{\theta}_{{\vec{m}}}})-f_{K}(\vec{{m},{\theta}_{{\vec{m}}}})=\frac{1}{2}(\vec{y-m})^{T} H_{\alpha,\vec{{m}}} (\vec{y-m})=f_{K}(\vec{{y-m},{\theta}_{{\vec{m}}}}).
\end{equation}
Moreover, for $\vec{y}=g\vec{m}$, we have the following lemma.

\begin{lemma}\label{CCcond}
For the positive mass vector $\vec{m}$, if
 \begin{equation*}\label{cond1}
 \vec{m}^{T}(g-I)^{T} H_{\alpha,\vec{{m}}} (g-I) \vec{m}<0,\ \ \emph{for some} \  g\in  G,
 \end{equation*}
 or 
 \begin{equation*}\label{cond2}
 f_{K}(g\vec{m},\vec{\theta_{m}})<f_{K}(\vec{m},\vec{\theta_{m}}),\ \ \emph{for some} \  g\in  G,
 \end{equation*}
   then there is no centered co-circular central configuration for $\vec{m}$.
\end{lemma}
\begin{proof}
We give the proof by contradiction.
Suppose $(\vec{{m},{\theta}_{{\vec{m}}}})\in \mathcal{CC}_{0}$, then the two conditions are equivalent by \eqref{eq4} in the above discussion.
However Lemma  \ref{SPCC} imply that 
$$f_{K}(\vec{m},\vec{\theta_{m}})=f_{K}(g\vec{m},\hat{g}\vec{\theta_{m}})\leq f_{K}(g\vec{m},\vec{\theta_{m}}),$$
 which is a contradiction.

\end{proof}

This lemma is our first criterion  to exclude any positive mass vector $\vec{m}$ which can not form  a centered co-circular CC, although it seems quite hard to compute $\vec{\theta_{m}}$ or the matrix $H_{\alpha,\vec{m}}$.

\begin{lemma}\label{switch}
Suppose n positive masses $m_{1},m_{2},\dots,m_{n}$ form a centered co-circular central configuration, then exchange the positions of any two unequal masses $m_{i},m_{j}$, the value of $f_{K}$ will strictly decrease. 
\end{lemma}
\begin{proof}
Suppose $\vec{m}=(m_{1},m_{2},\dots,m_{n})^{T}$ and $(\vec{m},\vec{\theta_{m}})\in\mathcal{CC}_{0}$.
Let
$$\vec{m'}=({m}_{1},\dots,{m}_{j-1},{m}_{k},{m}_{j+1},\dots,{m}_{k-1},{m}_{j},{m}_{k+1},\dots,m_{n}),$$
by \eqref{eq4} we have
$$\begin{aligned}
h_{\vec{m}}(\vec{m'})&=f_{K}(\vec{{m},{\theta}_{{m}}})+\frac{1}{2}(\vec{m'-{m}}) H_{\alpha,\vec{{m}}}  (\vec{m'-{m}})^{T},\\
&=f_{K}(\vec{{m},{\theta}_{{m}}})+({m}_{k}-{m}_{j})(r_{jk}^{-\alpha}+\frac{1}{K}r_{jk}^{2})({m}_{j}-{m}_{k}),
\end{aligned}$$
which implies 
 $$f_{K}(\vec{{m'},{\theta}_{{m}}})<f_{K}(\vec{{m},{\theta}_{{m}}}),\ \  \emph{if}\ \ \  {m}_{j}\neq{m}_{k}.$$
 \end{proof}
This lemma shows us some positive possibility of Chenciner's question because  the centered co-circular CC is some kind of  minimum by Lemma \ref{unique1}.


\begin{corollary}\label{n=3}
The equal-mass equilateral triangle CC is the unique centered co-circular central configuration for the general power-law potential $3$-body problem.
\end{corollary}
\begin{proof}
By Corollary \ref{em}, it is enough to show that all the masses are equal  and
we give the proof by contradiction. Suppose there is a 3-body centered co-circular central configuration $(\vec{m},\vec{\theta_{m}})\in \mathcal{CC}_{0}$ with unequal masses. 
Without loss of generality,  we suppose $m_{1}\neq m_{2}$, then  
$$\vec{m}=(m_{1},m_{2},m_{3})^{T}\neq (m_{2},m_{1},m_{3})^{T}=S\vec{m}.$$
Lemma \ref{switch} implies $f_{K}(S\vec{m},\vec{\theta_{m}})<f_{K}(\vec{m},\vec{\theta_{m}})$, but by Lemma \ref{CCcond} there is no centered co-circular CC for such $\vec{m}$, contradicts.
%
%
\end{proof}

\begin{corollary}
The equal-mass square CC is the unique centered co-circular central configuration for the general power-law potential $4$-body problem. 
\end{corollary}
\begin{proof}
Suppose $\vec{m}=(m_{1},m_{2},m_{3},m_{4})$ and $(\vec{m},\vec{\theta_{m}})\in \mathcal{CC}_{0},$
then $S\vec{m}=(m_{3},m_{2},m_{1},m_{4})$ and $P^{2}S\vec{m}=(m_{1},m_{4},m_{3},m_{2})$ can also form co-circular central configurations.
Applying Lemma \ref{CCcond} and \ref{switch}, a similar discussion to the proof of Corollary \ref{n=3} implies that $m_{1}=m_{3},m_{2}=m_{4}$, thus $\vec{m}=S\vec{m}=P^{2}S\vec{m}$. 
By Lemma \ref{SPCC}, we have 
$$\vec{\theta}=\mathcal{S}\vec{\theta}=\mathcal{P}^{2}\mathcal{S}\vec{\theta},$$
which implies $\vec{\theta}=(\frac{\pi}{2},\pi,\frac{3\pi}{2},2\pi)$, a square. 
It is an easy exercise to show that  the four masses should be equal to form a CC when the configuration is a square.
There is  also another more general result, Theorem 1 in \cite{Wang2019c},  $n\geq 4$ masses located at the vertices of a regular polygon form a central configuration with homogeneous potential $U_{\alpha} (\alpha\geq0)$ if and only if all the masses are equal. 
\end{proof}

It is more complicated when $n\geq5$ because the order of  masses will  change by just exchanging two unequal masses, unless the mass vector has some highly symmetry. 
Corbera and Valls \cite{CV2019} studies a special case where all the masses are equal except one for the general power-law potential $n$-body problem.  
Using Lemma \ref{CCcond} and \ref{switch}, we can have a much simpler proof.

\begin{corollary}
In the general power-law potential $n$-body problem, there are no centered  co-circular central configurations having all the masses equal except one. 
\end{corollary}
\begin{proof}
Without loss of generality, suppose the mass vector is 
$$\vec{m}=(1,\dots,1,m_{n}).$$
We give the proof by contradiction and suppose $m_{n}\neq1$ and $(\vec{m},\vec{\theta_m})\in \mathcal{CC}_{0}$.
Because
$(I-P)\vec{m}=(0,\dots,0,1-m_{n},m_{n}-1),$
then Lemma \ref{switch} implies  
 $$f_{K}(P\vec{{m},{\theta}_{{m}}})<f_{K}(\vec{{m},{\theta}_{{m}}}),$$
which contradicts with Lemma \ref{CCcond}.

\end{proof}

Moreover, we have
\begin{corollary}
In the general power-law potential $n$-body problem, when  $n$ is odd there are no centered  co-circular central configurations having all the masses equal except  two. 
\end{corollary}

\begin{proof}
Without loss of generality, suppose the mass vector is 
$$\vec{m}=(1,\dots,1,m_{k},1,\dots,1,m_{n}).$$
According to the previous corollary, it is enough to prove $m_{k}=1$ if $(\vec{m},\vec{\theta_m})\in \mathcal{CC}_{0}$.\\
Now suppose $m_{k}\neq1$ and $n$ is an odd number, we have
$$(I-S)\vec{m}=(0,\dots,0,m_{k}-1,0,\dots,0,1-m_{k},0,\dots,0,0).$$
Then by Lemma \ref{switch}, 
$$f_{K}(S\vec{{m},{\theta}_{{m}}})<f_{K}(\vec{{m},{\theta}_{{m}}}),$$
contradicts with Lemma \ref{CCcond}.
\end{proof}

\section{Proof of Theorem \ref{main1}}\label{sec4}


In this section we prove our main theorem.
We give the proof  by contradiction and throughout this proof we set $K=\frac{2^{3+\alpha}}{\alpha}$ for simplicity.
Suppose $(\vec{m},\vec{\theta_{m}})\in \mathcal{CC}_{0}$, 
and let $$\mathcal{H}_{\alpha,\vec{{m}}} =(\frac{2U_{\alpha}(\vec{m},\vec{\theta_{m}})}{M^{2}}+\frac{2}{K})J_{n}-H_{\alpha,\vec{{m}}} ,$$ 
where $J_{n}$ is the $n\times n$ matrix of ones which means  every element of $J_{n}$ is equal to one. 

\begin{lemma}\label{lemma12}
If $(\vec{m},\vec{\theta_{m}})\in \mathcal{CC}_{0}$ and $2^{\alpha+1}\frac{U_{\alpha}(\vec{m},\vec{\theta_{m}})}{M^{2}}\leq1+\frac{\alpha}{4}$, the symmetric matrix $\mathcal{H}_{\alpha,\vec{m}}$ is positive semi-definite  and the only zero eigenvector is $\vec{m}$.
\end{lemma}
\begin{proof}
Let the diagonal matrix $\mathcal{M}=diag\{m_{1},m_{2},\dots,m_{n}\},$ then to prove this lemma is equivalent to  prove that
the symmetric matrix $\mathcal{M}^{T}\mathcal{H}_{\alpha,\vec{m}}\mathcal{M}$ is positive semi-definite  and the only zero eigenvector is $\vec{1}$.
Since $(\vec{m},\vec{\theta_{m}})\in \mathcal{CC}_{0}$, then the second part of  \eqref{ccc} and \eqref{ccc1} imply that
$$H_{\alpha,\vec{m}}\vec{m}=(\frac{2U_{\alpha}(\vec{m},\vec{\theta_{m}})}{M}+\frac{2M}{K})\vec{1}.$$
Thus
$\mathcal{H}_{\alpha,\vec{{m}}}\vec{m}=\vec{0}$
which implies 
\begin{equation}\label{dominant}
\mathcal{M}^{T}\mathcal{H}_{\alpha,\vec{{m}}}\mathcal{M}\vec{1}=\vec{0}.
\end{equation} 
The condition $2^{\alpha+1}\frac{U_{\alpha}(\vec{m},\vec{\theta_{m}})}{M^{2}}\leq1+\frac{\alpha}{4}$ implies
$$\frac{2U_{\alpha}(\vec{m},\vec{\theta_{m}})}{M^{2}}\leq \frac{1}{2^{\alpha}}+\frac{\alpha}{2^{2+\alpha}}
= \frac{1}{2^{\alpha}}+\frac{2}{K},$$
and a direct computation shows that for $r_{ij}\in (0,2]$ and $K=\frac{2^{3+\alpha}}{\alpha}$,
$$\frac{1}{r_{ij}^{\alpha}}+\frac{r_{ij}^{2}}{K}\geq \frac{1}{2^{\alpha}}+\frac{4}{K}.$$
Thus the off-diagonal entries  of the matrix $\mathcal{M}^{T}\mathcal{H}_{\alpha,\vec{{m}}}\mathcal{M}$ 
 $$m_{i}m_{j}[(\frac{2U_{\alpha}(\vec{m},\vec{\theta_{m}})}{M^{2}}+\frac{2}{K})-(\frac{1}{r_{ij}^{\alpha}}+\frac{r_{ij}^{2}}{K})]\leq0.$$
Together with \eqref{dominant}, we see that the matrix  $\mathcal{M}^{T}\mathcal{H}_{\alpha,\vec{{m}}}\mathcal{M}$  is diagonally dominant and thus positive semi-definite.
Moreover the vector  $\vec{1}$ is its unique zero eigenvector(see Remark \ref{remark1}).
%

\end{proof}

This lemma implies that, for unequal positive masses $\vec{m}$ and  $2^{\alpha+1}\frac{U_{\alpha}(\vec{m},\vec{\theta_{m}})}{M^{2}}\leq1+\frac{\alpha}{4}$,  $$\vec{m}^{T}(P-I)^{T} H_{\alpha,\vec{{m}}} (P-I) \vec{m}<0.$$  
Applying Lemma \ref{CCcond}, we have the following corollary.

\begin{corollary}
A necessary condition for unequal positive masses $\vec{m}$ to form a centered co-circular central configuration, namely $(\vec{m},\vec{\theta_{m}})\in \mathcal{CC}_{0},$
is that 
$$2^{\alpha+1}\frac{U_{\alpha}(\vec{m},\vec{\theta_{m}})}{M^{2}}>1+\frac{\alpha}{4}.$$
\end{corollary}

Considering $\vec{\theta_{m}}$ is hard to compute and noticing that $\vec{\theta_{m}}$ is a minimum, we have the following corollary which can be  seen as our second criterion to judge whether a  masses vector $\vec{m}$ can not form a centered co-circular CC.

\begin{corollary}\label{cri}
For any given positive unequal masses $\vec{m}$, if there is some $\vec{\varphi}\in\mathcal{K}_{0}$ such that
$$2^{\alpha+1}\frac{U_{\alpha}(\vec{m},\vec{\varphi})}{M^{2}}\leq1+\frac{\alpha}{4},$$
then it can not be  a centered co-circular central configuration, namely $(\vec{m},\vec{\theta_{m}})\not\in \mathcal{CC}_{0}.$

\end{corollary}

We have already known that, the regular $n$-gon CC with equal masses is the trivial centered co-circular central configuration.
When we consider the regular $n$-gon CC, it is interesting to notice that $\mathcal{H}_{\alpha,\vec{1}},J_{n},H_{\alpha,\vec{1}}$ are all circulant matrices \cite{marcus1992survey}. 

An $n\times n$ matrix $C=(c_{kj})$ is called circulant if
$c_{kj}=c_{k-1,j-1}$ 
where $c_{0,j}$  and $c_{k,0}$  are identified with $c_{n,j}$ and $c_{k,n}$, respectively.
Thus every circulant matrix $C$ can be represented as
$C=\sum_{j=1}^nc_{1j}P^{j-1},$
 then they all have the same eigenvectors     $\vec{v}_k=(\xi_{k-1},\xi_{k-1}^2,\cdots,\xi_{k-1}^n)^T$, where $\xi_k=e^{\frac{2k\pi}{n}\sqrt{-1}}$  denote the $n$ complex $n$th roots of unity.
    
So  considering the regular $n$-gon CC with equal masses $(\vec{1},\vec{\theta_{1}})\in\mathcal{CC}_{0}$,
if 
$$2^{\alpha+1}\frac{U_{\alpha}(\vec{1},\vec{\theta_{1}})}{n^{2}}=\frac{1}{n}\sum_{j=1}^{n-1}\csc^{\alpha}\frac{j\pi}{n}\leq1+\frac{\alpha}{4},$$
then by Lemma \ref{lemma12}, $\mathcal{H}_{\alpha,\vec{1}}$ will be positive semi-definite with only one zero eigenvalue.
This implies that
the eigenvalues of $H_{\alpha,\vec{1}} $ are negative except the first one corresponding to the eigenvector $\vec{v}_{1}=\vec{1}$. So for all unequal masses $\vec{m}$, we have
 $$\frac{\vec{m}^{T}H_{\alpha,\vec{1}}\vec{m} }{M^{2}}< \frac{\vec{1}^{T}H_{\alpha,\vec{1}}\vec{1} }{n^{2}}.$$ 
 We notice that $f_{K}(\vec{m},\vec{\theta_{m}})=\frac{1}{2}\vec{m}^{T}H_{\alpha,\vec{m}}\vec{m}$, so by Lemma \ref{unique1}
 $$\frac{\vec{m}^{T}H_{\alpha,\vec{m}}\vec{m} }{M^{2}}\leq\frac{\vec{m}^{T}H_{\alpha,\vec{1}}\vec{m} }{M^{2}}< \frac{\vec{1}^{T}H_{\alpha,\vec{1}}\vec{1} }{n^{2}}.$$
Suppose there exist some unequal masses  $\vec{m}$ to form a centered co-circular CC, namely $(\vec{m},\vec{\theta_{m}})\in \mathcal{CC}_{0}$, the above inequality implies
$$\frac{2U_{\alpha}(\vec{m},\vec{\theta_{m}}) +2M^{2}}{M^{2}}< \frac{2U_{\alpha}(\vec{1},\vec{\theta_{1}}) +2n^{2} }{n^{2}},$$ 
then
$$2^{\alpha+1}\frac{U_{\alpha}(\vec{m},\vec{\theta_{m}}) }{M^{2}}< 1+\frac{\alpha}{4}.$$ 
So Theorem \ref{main1} follows by Corollary \ref{cri}.

\section{Acknowledgment}  The  author is supported by  the National Natural Science Foundation of China 
and China Scholarship Council. 
Part of this work was done when he was visiting University of Minnesota.  He thanks the School of Mathematics and Professor Richard Moeckel for their hospitality and support.

\end{document}